\documentclass[hidelinks]{siamart190516}

\usepackage{amsmath}
\usepackage{amssymb}

\title{A measure concentration effect for matrices 
  of high, higher, and even higher dimension}

\author{Harry Yserentant
 \thanks{TU Berlin, Institut f\"ur Mathematik, 10623 Berlin, Germany 
         (yserentant@math.tu-berlin.de)}} 

\begin{document}

\maketitle

\begin{abstract}
Let $n>m$, and let $A$ be an $(m\times n)$-matrix of full rank.
Then obviously the estimate $\|Ax\|\leq\|A\|\|x\|$ holds for 
the euclidean norm of $x$ and $Ax$ and the spectral norm as 
the assigned matrix norm. We study the sets of all $x$ for which, 
for fixed $\delta<1$, conversely $\|Ax\|\geq\delta\,\|A\|\|x\|$ 
holds. It turns out that these sets fill, in the 
high-dimensional case, almost the complete space once $\delta$ 
falls below a bound that depends on the extremal singular 
values of $A$ and on the ratio of the dimensions. This effect 
has much to do with the random projection theorem, which plays
an important role in the data sciences. As a byproduct, we 
calculate the probabilities this theorem deals with exactly.
\end{abstract}



%
%
%
%
%

\pagestyle{myheadings}
\thispagestyle{plain}

\markboth{\uppercase{Harry Yserentant}}
{\uppercase{A measure concentration effect for matrices}}


\renewcommand {\thefigure}{\arabic{figure}}

\newcommand   {\rmref}[1]   {{\rm (\ref{#1})}}

\newcommand   {\diff}[1]    {\mathrm{d}#1}

\newcommand   {\fourier}[1] {\widehat{#1}}

\def \xy      {\Big(\,\begin{matrix}x\\y\end{matrix}\,\Big)}

\def \dx      {\,\diff{x}}
\def \dy      {\,\diff{y}}

\def \dxy     {\,\diff{(x,y)}}

\def \deta    {\,\diff{\eta}}
\def \domega  {\,\diff{\omega}}

\def \dr      {\,\diff{r}}
\def \ds      {\,\diff{s}}
\def \dt      {\,\diff{t}}

\def \L       {\mathcal{L}}


\section{Introduction}
\label{sec1}

\setcounter{equation}{0}
\setcounter{theorem}{0}

Let $n>m$ and let $A$ be a real $(m\times n)$-matrix
of rank $m$. The kernel of such a matrix has the 
dimension $n-m$ and hence can, in dependence of the 
dimensions, be a large subspace of the $\mathbb{R}^n$. 
Nevertheless, the set of all $x$ for which
\begin{equation}    \label{eq1.1}
\|Ax\|\geq\delta\,\|A\|\|x\|
\end{equation}
holds fills, in the high-dimensional case, often 
almost the complete $\mathbb{R}^n$ once $\delta$ 
falls below a certain bound; the involved norms 
are here and throughout the paper the euclidean 
norm on the $\mathbb{R}^m$ and the $\mathbb{R}^n$  
and the assigned spectral norm of matrices. Let 
$\chi$ be the characteristic function of the 
set of all $x$ for which  
\mbox{$\|Ax\|<\delta\,\|A\|\|x\|$} holds, and 
let $\nu_n$ be the volume of the unit ball in 
$\mathbb{R}^n$. The normed area measure
\begin{equation}    \label{eq1.2}
\frac{1}{n\nu_n}\int_{S^{n-1}}\!\chi(\eta)\deta
\end{equation}
of the subset of the unit sphere on which the 
condition (\ref{eq1.1}) is violated takes in 
such cases an extremely small value, which
conversely again means that (\ref{eq1.1}) holds 
on an overwhelmingly large part of the unit 
sphere and with that of the full space. The 
aim of the present paper is to study this 
phenomenon in dependence of characteristic 
quantities like the ratio of the dimensions 
$m$ and $n$ and the extremal singular values 
of the matrices qualitatively and, as far as 
possible, also quantitatively. 

The described effect is best understood for orthogonal 
projections. For matrices of this kind, this observation 
is a direct consequence of the random projection theorem 
(see Lemma~5.3.2 in \cite{Vershynin}, for example), which 
is in close connection with the Johnson--Lindenstrauss 
theorem \cite{Johnson-Lindenstrauss}. The random projection 
theorem deals with orthogonal projections from the 
$\mathbb{R}^n$ onto random subspaces of lower dimension 
$m$, but equally one can consider orthogonal projections 
$Px$ of random vectors $x$ to the~$\mathbb{R}^m$. The 
theorem states that with probability greater than 
$1-2\exp(-c\,\varepsilon^2m)$
\begin{equation}    \label{eq1.3}
(1-\varepsilon)\sqrt{\frac{m}{n}}\,\|x\|
\leq \|Px\| \leq
(1+\varepsilon)\sqrt{\frac{m}{n}}\,\|x\|
\end{equation}
holds for all $x$ on the unit sphere and thereby 
also in the full space. The random projection 
theorem is a manifestation of the concentration 
of measure phenomenon, which plays a fundamental 
role in the analysis of very many problems in 
high space dimensions and became a backbone of 
high-dimensional probability theory and modern 
data science. The interest in the concentration 
of measure phenomenon arose in the early 1970s 
in the study of the asymptotic theory of Banach 
spaces. Classical texts are \cite{Ledoux} and 
\cite{Ledoux-Talagrand}. An up-to-date exposition 
containing a lot of information on random vectors 
and matrices is \cite{Vershynin}. In the present 
article, we carefully reconsider the random 
projection theorem. Among other things, we 
calculate the normed area measure (\ref{eq1.2}) 
for projection matrices of the described kind 
for even-numbered differences of the two 
dimensions exactly and derive a very sharp 
inclusion for the odd-numbered case. The 
results for these projection matrices serve 
as a basis for the examination of general 
matrices in dependence of their singular values. 

One might ask what is known about these. The 
simple answer is that this depends  on the 
class of matrices one considers. Graph theory 
\cite{Brouwer-Haemers}, 
\cite{Cvetkovic-Rowlinson-Simic} is an
important source of information. There is 
an extensive literature about the extremal 
singular values of random matrices with 
independent, identically distributed 
entries. An early breakthrough was Edelman's 
thesis \cite{Edelman}. Other significant 
contributions are 
\cite{Rudelson-Vershynin}, \cite{Tao-Vu}, 
and, very recently, 
\cite{Livshyts-Tikhomirov-Vershynin}. 
Random matrices play an important role 
in fields like compressed sensing 
\cite{Foucart-Rauhut} and all sorts 
of data acquisition and compression 
techniques. Our interest in the problem 
originates from the attempt \cite{Yserentant} 
to extend the applicability of modern tensor 
product methods \cite{Hackbusch} to more 
general problem classes. Assume that we 
are looking for the solution 
$u:\mathbb{R}^m\to\mathbb{R}$ of the 
Laplace-like equation 
\begin{equation}    \label{eq1.4}
-\Delta u+\mu u=f
\end{equation}
that vanishes at infinity, where $\mu>0$ is
constant and the right-hand side $f$ is, 
for instance, a product of functions depending 
only on a single component of $x$ or on the 
difference of two such components. The 
question is how well such structures are 
reflected in the solution of the equation.

Let us assume that the right-hand side is 
of the form $f(x)=F(Tx)$, with a function
$F:\mathbb{R}^n\to\mathbb{R}$, $n>m$, with
an integrable Fourier transform and with $T$ 
a matrix of full rank that is determined by 
the structure of the underlying problem. As 
shown in \cite{Yserentant}, the solution 
is then the trace $u(x)=U(Tx)$ of the 
function
\begin{equation}    \label{eq1.5}
U(y)=\Big(\frac{1}{\sqrt{2\pi}}\Big)^n\!\int
\frac{1}{\mu+\|T^t\omega\|^2}\,\fourier{F}(\omega)\,
\mathrm{e}^{\,\mathrm{i}\,\omega\cdot y}\domega.
\end{equation}
The function (\ref{eq1.5}) is in the domain 
of the operator $\L$ given by
\begin{equation}    \label{eq1.6}
(\L U)(y)=\Big(\frac{1}{\sqrt{2\pi}}\Big)^n\!\int
\big(\mu+\|T^t\omega\|^2\big)\fourier{U}(\omega)\,
\mathrm{e}^{\,\mathrm{i}\,\omega\cdot y}\domega
\end{equation}
and solves by definition the degenerate
second-order elliptic equation $\L U=F$.
It can be calculated approximately by 
means of the iteration
\begin{equation}    \label{eq1.7}
U_{k+1}=(I-\alpha\L)U_k+\alpha F,
\end{equation}
where the operator $\alpha$ is given by
\begin{equation}    \label{eq1.8}
(\alpha F)(y)=\Big(\frac{1}{\sqrt{2\pi}}\Big)^n\!\int
\frac{1}{\mu+\|T^t\|^2\|\omega\|^2}\,
\fourier{F}(\omega)\,
\mathrm{e}^{\,\mathrm{i}\,\omega\cdot y}\domega.
\end{equation}
Provided that $\|T^t\omega\|\geq\delta\,\|T^t\|\|\omega\|$
holds on the support of $\fourier{F}$, the $L_1$-norm
of the Fourier transform of the error is in every
iteration step reduced by the factor \mbox{$1-\delta^2$},
or by more than the factor $1-\delta$ with polynomial 
acceleration. If this condition is only violated on a 
very small set, the additional error can in general be 
neglected without hard conditions to $\fourier{F}$ or 
$F$ itself. The idea is to approximate the kernel in 
(\ref{eq1.8}) by a linear combination of Gauss functions. 
If $F$ is as in the example above the product of 
lower-dimensional functions depending only on small 
groups of components of $Tx$, the iterates are then 
composed of functions of the same type.


\section{Reformulations as volume integrals and first estimates}
\label{sec2}

\setcounter{equation}{0}
\setcounter{theorem}{0}

The surface integrals (\ref{eq1.2}) are not easily 
accessible and are difficult to calculate and estimate. 
We reformulate them therefore at first as volume 
integrals and draw some first conclusions from 
these representations. The starting point is the 
decomposition
\begin{equation}    \label{eq2.1}
\int_{\mathbb{R}^n}f(x)\dx\,= 
\int_{S^{n-1}}\left(\int_0^\infty\!f(r\eta)r^{n-1}\dr\right)\!\deta
\end{equation}
of the integrals of functions in $L_1$ into 
an inner radial and an outer angular part. 
Inserting the characteristic function of the
unit ball, one recognizes that the area of
the $n$-dimensional unit sphere is $n\nu_n$,
with $\nu_n$ the volume of the unit ball. If 
$f$ is rotationally symmetric, $f(r\eta)=f(re)$ 
holds for every $\eta\in S^{n-1}$ and every 
fixed, arbitrarily given unit vector $e$.
In this case, (\ref{eq2.1}) reduces 
therefore to
\begin{equation}    \label{eq2.2}
\int f(x)\dx\,=n\nu_n\int_0^\infty\!f(re)r^{n-1}\dr.
\end{equation}
The volume measure on the $\mathbb{R}^n$ will 
in the following be denoted by $\lambda$.

\begin{lemma}       \label{thm2.1}
Let $A$ be an arbitrary matrix of dimension $m\times n$, 
$m<n$, let $\chi$ be the characteristic function of
the set of all $x\in\mathbb{R}^n$ for which 
$\|Ax\|<\delta\,\|A\|\|x\|$ holds, and let 
$W:\mathbb{R}^n\to\mathbb{R}$ be a rotationally 
symmetric function with integral
\begin{equation}    \label{eq2.3}
\int W(x)\dx\,=\,1.
\end{equation}
The weighted surface integral \rmref{eq1.2}  
then takes the value
\begin{equation}    \label{eq2.4}
\int \chi(x)W(x)\dx.
\end{equation}
\end{lemma}

\begin{proof}
Let $e$ be a given unit vector. For $\eta\in S^{n-1}$ and 
$r>0$ then $\chi(r\eta)=\chi(\eta)$ and $W(r\eta)=W(re)$ 
holds and the integral (\ref{eq2.4}) can by (\ref{eq2.1})
be written as
\begin{displaymath}
\int\chi(x)W(x)\dx\,= 
\int_{S^{n-1}}\chi(\eta)\left(\int_0^\infty\!W(re)r^{n-1}\dr\right)\!\deta.
\end{displaymath}
Because the inner integral takes by (\ref{eq2.2}) 
and (\ref{eq2.3}) the value
\begin{displaymath}
\int_0^\infty\!W(re)r^{n-1}\dr=\frac{1}{n\nu_n},
\end{displaymath}
this proves the proposition.
\end{proof}

An obvious choice for the weight function $W$, 
which will later still play an important role
and will be used at several places, is the 
normed Gauss function
\begin{equation}    \label{eq2.5}
W(x)=\Big(\frac{1}{\sqrt{\pi}}\Big)^n\exp\big(-\|x\|^2\big).
\end{equation}
Another possible choice is the characteristic 
function of the ball of radius $R$ around the 
origin divided by the volume of this ball. 
It leads to the following lemma.

\begin{lemma}       \label{thm2.2}
Let $A$ be a matrix of dimension $m\times n$, $m<n$, 
and let $\lambda$ be the volume measure on the
$\mathbb{R}^n$. The weighted integral \rmref{eq1.2} 
over the surface of the unit ball is then 
independent of the radius $R$ equal to the volume 
ratio
\begin{equation}    \label{eq2.6}
\frac
{\lambda\big(\big\{x\,\big|\,\|Ax\|<\delta\,\|A\|\|x\|,\,\|x\|\leq R\big\}\big)}
{\lambda\big(\big\{x\,\big|\,\|x\|\leq R\big\}\big)}.
\end{equation}
\end{lemma}

Because the euclidean length of a vector and 
the volume of a set are invariant to orthogonal 
transformations, the surface ratio (\ref{eq1.2}) 
and the volume ratio (\ref{eq2.6}) as well 
depend only on the singular values of the 
matrix under consideration.

\begin{lemma}       \label{thm2.3}
Let $A$ be a matrix of dimension $m\times n$, $m<n$, 
with singular value decomposition $A=U\Sigma V^t$.
The volume ratios \rmref{eq2.6} are then equal to 
the volume ratios
\begin{equation}    \label{eq2.7}
\frac
{\lambda\big(\big\{x\,\big|\,\|\Sigma x\|<\delta\,\|\Sigma\|\|x\|,\,\|x\|\leq R\big\}\big)}
{\lambda\big(\big\{x\,\big|\,\|x\|\leq R\big\}\big)}; 
\end{equation}
that is, they depend exclusively on the singular 
values of the matrix $A$.
\end{lemma}

\begin{proof}
As the multiplication with the orthogonal matrices 
$U$ and $V^t$, respectively, does not change the 
euclidean norm of a vector, the set of all 
$x\in\mathbb{R}^n$ for which 
\begin{displaymath}
\|Ax\|<\delta\,\|A\|\|x\|,\quad \|x\|\leq R,
\end{displaymath}
holds coincides with the set of all $x$ for which 
we have
\begin{displaymath}
\|\Sigma V^tx\|<\delta\,\|\Sigma\|\|V^tx\|,
\quad \|V^tx\|\leq R.
\end{displaymath}
As the volume is invariant to orthogonal 
transformations, the proposition follows.
\end{proof}

Orthogonal projections, or in other words matrices 
with one as the only singular value, represent one of 
the few cases for which the volume ratios (\ref{eq2.6})
can be more or less explicitly calculated. Orthogonal 
projections are of particular importance and will, 
as said, serve as the anchor for many of our estimates. 
Again, it suffices to consider the corresponding 
diagonal matrices $\Sigma$, denoted in the 
following by $P$.

\begin{theorem}     \label{thm2.4}
Let $P$ be the $(m\times n)$-matrix that extracts 
from a vector in $\mathbb{R}^n$ its first $m$ 
components. For $0\leq\delta<1$ and all radii 
$R>0$, then
\begin{equation}    \label{eq2.8}
\frac{\lambda\big(
\big\{x\,\big|\,\|Px\|<\delta\,\|x\|,\,\|x\|\leq R\big\}\big)}
{\lambda\big(\big\{x\,\big|\,\|x\|\leq R\big\}\big)}
=\psi\bigg(\frac{\delta}{\sqrt{1-\delta^2}}\bigg)
\end{equation}
holds, where the function $\psi$ is defined 
by the integral expression
\begin{equation}    \label{eq2.9}
\psi(\varepsilon)=
\frac{2\,\Gamma(n/2)}{\Gamma(m/2)\Gamma((n-m)/2)}
\int_0^{\varepsilon}\frac{t^{m-1}}{(1+t^2)^{n/2}}\dt.
\end{equation}
\end{theorem}

\begin{proof}
Differing from the notation in the theorem 
but consistent within the proof, we split 
the vectors in $\mathbb{R}^n$ into parts 
$x\in\mathbb{R}^m$ and $y\in\mathbb{R}^{n-m}$.
The set whose volume has to be calculated 
consists then of the points in the given 
ball for which
\begin{displaymath}
\|x\|<\delta\,\Big\|\xy\Big\|
\end{displaymath}
or, resolved for the norm of the component
$x\in\mathbb{R}^m$,
\begin{displaymath}
\|x\|<\varepsilon\,\|y\|, \quad
\varepsilon=\frac{\delta}{\sqrt{1-\delta^2}},
\end{displaymath}
holds. For homogeneity reasons, that is, by
Lemma~\ref{thm2.2}, we can restrict ourselves 
to the ball of radius $R=1$. The volume can 
then be expressed as double integral 
\begin{displaymath}
\int\bigg(\int H\big(\varepsilon\|y\|-\|x\|\big)
\chi\big(\|x\|^2+\|y\|^2\big)\dx\bigg)\dy,
\end{displaymath}
where $H(t)=0$ for $t\leq0$, $H(t)=1$ for $t>0$,
$\chi(t)=1$ for $t\leq 1$, and $\chi(t)=0$ for 
arguments $t>1$. In terms of polar coordinates, 
that is, by (\ref{eq2.2}), it reads as
\begin{displaymath}
(n-m)\nu_{n-m}\int_0^\infty\bigg(
m\nu_m\int_0^{\varepsilon s}\chi\big(r^2+s^2\big)r^{m-1}\dr
\bigg)s^{n-m-1}\ds,
\end{displaymath}
with $\nu_d$ the volume of the $d$-dimensional unit 
ball. Substituting $t=r/s$ in the inner integral,
the upper bound becomes independent of $s$ and 
the integral can be written~as
\begin{displaymath}
(n-m)\nu_{n-m}\int_0^\infty\bigg(m\nu_m\,s^m\!
\int_0^{\varepsilon}\chi\big(s^2(1+t^2)\big)t^{m-1}\dt
\bigg)s^{n-m-1}\ds,
\end{displaymath}
and interchanging the order of integration, 
it attains finally the value
\begin{displaymath}
\frac{(n-m)\nu_{n-m}\,m\nu_m}{n}
\int_0^{\varepsilon}\frac{t^{m-1}}{(1+t^2)^{n/2}}\dt.
\end{displaymath}
Dividing this by the volume $\nu_n$ of the 
unit ball itself and remembering that 
\begin{displaymath}
\nu_d=\frac{2}{d}\,\frac{\pi^{d/2}}{\Gamma(d/2)},
\end{displaymath}
this completes the proof of the theorem.
\end{proof}

The following lemma describes the dependence of the 
volume ratio (\ref{eq2.8}) on the dimensions $m$ and 
$n$. In conjunction with Theorem~\ref{thm3.2} below 
it can be used to enclose the volume ratio from 
both sides for uneven differences of the dimensions.

\begin{lemma}       \label{thm2.5}
The volume ratio \rmref{eq2.8} decreases, for $n$ kept 
fixed, when $m$ increases, and it increases, for $m$ 
kept fixed, when $n$ increases.
\end{lemma}

\begin{proof}
The set in the numerator on the left-hand side of (\ref{eq2.8}) 
gets smaller when $m$ gets larger. This proves the first 
proposition. The argumentation for increasing dimension $n$ 
is more involved. It is based on the representation from 
Lemma~\ref{thm2.1} with the weight function  (\ref{eq2.5}). 
For $x\in\mathbb{R}^n$ and $y\in\mathbb{R}^p$, let
\begin{displaymath}
\chi(x,y)=
\begin{cases}
\,1,& \text{if $\|Px\|^2<\delta^2\big(\|x\|^2+\|y\|^2\big)$}\\
\,0,& \text{otherwise}.
\end{cases}
\end{displaymath}
By Lemma~\ref{thm2.1}, the volume ratio (\ref{eq2.8}) 
can then be written as the integral
\begin{displaymath} 
\Big(\frac{1}{\sqrt{\pi}}\Big)^n\!
\int\chi(x,0)\exp\big(-\|x\|^2\big)\dx
\end{displaymath}
over the $\mathbb{R}^n$. This integral takes
by Fubini's theorem the same value as the
integral
\begin{displaymath} 
\Big(\frac{1}{\sqrt{\pi}}\Big)^{n+p}\!
\int\chi(x,0)\exp\big(-\big(\|x\|^2+\|y\|^2\big)\big)\dxy
\end{displaymath}
over the $\mathbb{R}^n\times\mathbb{R}^p$ and 
can be estimated from above by the integral
\begin{displaymath} 
\Big(\frac{1}{\sqrt{\pi}}\Big)^{n+p}\!
\int\chi(x,y)\exp\big(-\big(\|x\|^2+\|y\|^2\big)\big)\dxy.
\end{displaymath}
This integral takes by Lemma~\ref{thm2.1} the same value 
as the volume ratio (\ref{eq2.8}) , with the original 
dimension $m$, but with the dimension $n$ now 
replaced by $n'=n+p$.
\end{proof}

\pagebreak  

The volume ratios (\ref{eq2.6}) can be bounded 
from both sides in terms of the, as we will see, 
more or less explicitly known volumes ratios 
(\ref{eq2.8}), i.e., of the function (\ref{eq2.9}).

\begin{theorem}     \label{thm2.6}              
Let $n>m$, let $A$ be an $(m\times n)$-matrix of full
rank, and let $\kappa$ be the condition number of 
$A$, the ratio of its maximum to its minimum singular 
value. If $\kappa\delta$ is less than one, then for 
all radii $R>0$ the upper estimate
\begin{equation}    \label{eq2.10}
\frac
{\lambda\big(\big\{x\,\big|\,\|Ax\|<\delta\,\|A\|\|x\|,\,\|x\|\leq R\big\}\big)}
{\lambda\big(\big\{x\,\big|\,\|x\|\leq R\big\}\big)} 
\leq 
\psi\bigg(\frac{\kappa\delta}{\sqrt{1-\kappa^2\delta^2}}\bigg)
\end{equation}
holds, where $\psi$ is the function \rmref{eq2.9}.
Without further conditions to $\delta<1$, conversely
\begin{equation}    \label{eq2.11}
\frac
{\lambda\big(\big\{x\,\big|\,\|Ax\|<\delta\,\|A\|\|x\|,\,\|x\|\leq R\big\}\big)}
{\lambda\big(\big\{x\,\big|\,\|x\|\leq R\big\}\big)} 
\geq 
\psi\bigg(\frac{\delta}{\sqrt{1-\delta^2}}\bigg).
\end{equation}
holds. For orthogonal projections, that is, 
if $\kappa=1$, in both cases equality holds.
\end{theorem}

\begin{proof}
We can restrict ourselves in the proof to the diagonal
matrices $\Sigma$ from Lemma~\ref{thm2.3}. The proposition 
follows then rather immediately from the inequalities
\begin{displaymath}
\sigma_1\|Px\|\leq\|\Sigma x\|\leq\sigma_m\|Px\|
\end{displaymath}
and the fact that $\|\Sigma\|=\sigma_m$ comparing 
the corresponding volumes.
\end{proof}

Undoubtedly, (\ref{eq2.10}) and (\ref{eq2.11}) 
are despite their generality in many cases 
rather poor estimates because they largely ignore 
the underlying geometry. If the singular values 
$\sigma_1\leq\sigma_2\leq\cdots\leq\sigma_m$
of the matrix $A$ cluster around $\sigma_m>0$
or are, up to very few, even equal to $\sigma_m$, 
the following lemma opens a way out and 
provides a remedy.

\begin{lemma}       \label{thm2.7}
Let $\sigma_k=\sigma_m$ for all $k>m_0$ and let 
$P'$ be the matrix that extracts from a vector 
$x\in\mathbb{R}^n$ its components $x_k$, 
$m_0<k\leq m$. The volume ratio \rmref{eq2.6} 
is then less than or at most equal to the 
volume ratio
\begin{equation}    \label{eq2.12}
\frac
{\lambda\big(\big\{x\,\big|\,\|P'x\|<\delta\,\|x\|,\,\|x\|\leq R\big\}\big)}
{\lambda\big(\big\{x\,\big|\,\|x\|\leq R\big\}\big)}. 
\end{equation}
\end{lemma}

\begin{proof}
This follows by Lemma~\ref{thm2.3} from 
$\sigma_m\|P'x\|\leq\|\Sigma x\|$ 
and $\|\Sigma\|=\sigma_m$.
\end{proof}

The volume ratio (\ref{eq2.12}) possesses then
a representation like that in Theorem~\ref{thm2.4},
where $m'=m-m_0$ replaces the dimension $m$. But 
above all the potentially disastrous influence of 
the condition number vanishes. As indicated, the 
argumentation can be generalized to the case that  
the ratio $\kappa'=\sigma_m/\sigma_k$ is small in 
comparison to $\kappa$ for an index $k=m_0+1$ 
that is small in comparison to $m$.

The example that we have here in mind arises 
in connection with the iterative solution of 
high-dimensional elliptic partial differential 
equations as sketched in the introduction. 
The dimensions of the matrices $A=T^t$ under 
consideration are  
\begin{equation}    \label{eq2.13}
m=3\times N, \quad n=3\times\frac{N(N+1)}{2}. 
\end{equation}
The vectors $x$ in $\mathbb{R}^{m}$ and 
$\mathbb{R}^{n}$, respectively, are partitioned
into subvectors $x_i\in\mathbb{R}^{3}$. The 
matrices $T$ map the parts $x_i$ of 
$x\in\mathbb{R}^m$ first to themselves and 
then to the differences $x_i-x_j$, $i<j$. 
If one thinks of the Schr\"odinger equation, 
the $x_i$ are associated with the positions 
of $N$ electrons or other particles. The 
structure of $T$ reflects then that approximate 
solutions are sought that are composed of 
products of orbitals, depending only on a 
single component $x_i$, and of geminals, 
functions of the differences $x_i-x_j$. 
The euclidean norm of the vector 
$Tx\in\mathbb{R}^n$ is given by
\begin{equation}    \label{eq2.14}
\|Tx\|^2=\sum_{i=1}^N\|x_i\|^2+
\frac12\sum_{i=1}^N\sum_{j=1}^N\|x_i-x_j\|^2
\end{equation}
or, after rearrangement, with the rank 
three map $T_0x=x_1+x_2+\cdots+x_m$ by
\begin{equation}    \label{eq2.15}
\|Tx\|^2=(N+1)\|x\|^2-\|T_0x\|^2.
\end{equation}
The square matrix $T^tT$ therefore has the 
eigenvalues $1$ and $N+1$ and only the 
first three singular values of the matrix 
$T^t$ differ from the last one. If only some 
of the differences $x_i-x_j$ are taken into 
account, the minimum singular value of the 
resulting matrix remains $\sigma_1=1$ and 
the maximum singular value and the ratio 
of the dimensions as well can be bounded 
in terms of the degrees of the vertices 
of the underlying graph \cite{Yserentant}. 
The spectral theory of graphs is itself 
a large field \cite{Brouwer-Haemers}, 
\cite{Cvetkovic-Rowlinson-Simic} of great
importance and has numerous applications.


\section{Exact representations for orthogonal projections}
\label{sec3}

\setcounter{equation}{0}
\setcounter{theorem}{0}

One of the primary aims of this paper is a detailed 
study of the volume ratio (\ref{eq2.8}), 
\begin{equation}    \label{eq3.1}
\psi\bigg(\frac{\delta}{\sqrt{1-\delta^2}}\bigg),
\quad 0\leq\delta<1,
\end{equation}
and of its limit behavior when the dimensions tend 
to infinity. The starting point is at first an 
integral representation of this expression, which 
can also serve as a basis for its approximate 
calculation via a quadrature formula.

\begin{theorem}     \label{thm3.1}
The expression \rmref{eq3.1} possesses for 
$0\leq\delta<1$ the representation
\begin{equation}    \label{eq3.2}
\psi\bigg(\frac{\delta}{\sqrt{1-\delta^2}}\bigg)
=\frac{2\,\Gamma(n/2)}{\Gamma(m/2)\Gamma((n-m)/2)}\,
\int_0^\delta(1-t^2)^\alpha t^{m-1}\dt,
\end{equation}
where the exponent $\alpha\geq -1/2$ is given by
\begin{equation}    \label{eq3.3}
\alpha=\frac{n-m-2}{2}
\end{equation}
and takes nonnegative values for dimensions
$n\geq m+2$.
\end{theorem}

\begin{proof}
For abbreviation, we introduce the function 
\begin{displaymath}
f(\delta)=g\bigg(\frac{\delta}{\sqrt{1-\delta^2}}\bigg),
\quad
g(\varepsilon)=
\int_0^{\varepsilon}\frac{t^{m-1}}{(1+t^2)^{n/2}}\dt,
\end{displaymath}
on the interval $0\leq\delta<1$. Its derivative 
is the continuous function
\begin{displaymath}
f'(\delta)=(1-\delta^2)^\alpha\delta^{m-1}.
\end{displaymath}
Because $f(0)=0$, it possesses therefore the 
representation
\begin{displaymath}
f(\delta)=\int_0^\delta(1-t^2)^\alpha t^{m-1}\dt.
\end{displaymath}
This already proves the proposition.
\end{proof}

The integral (\ref{eq3.2}) can by means of the 
substitution $t=\sin\varphi$  be transformed 
into an integral over a trigonometric polynomial 
and can thus be calculated in terms of elementary 
functions. This does, however, not help too much 
because of the inevitably arising cancellation 
effects as soon as one tries to evaluate the 
result numerically. Such problems can be avoided 
if $\alpha$ is an integer, that is, if $m$ and 
$n$ are both even or both odd. The volume ratio 
(\ref{eq3.1}) is then a polynomial in $\delta$ 
that is composed of positive terms, which can as 
such be summed up in a numerically stable way.

\begin{theorem}     \label{thm3.2}
If the difference $n-m$ of the dimensions
is even, the function
\begin{equation}    \label{eq3.4}
\psi\bigg(\frac{\delta}{\sqrt{1-\delta^2}}\bigg) 
=\;\sum_{j=0}^{k}\,
\frac{\Gamma(k+l+1)}{\Gamma(k-j+1)\Gamma(l+j+1)}\,
(1-\delta^2)^{k-j}\,\delta^{2(l+j)}
\end{equation}
is a polynomial of degree $n-2$ in $\delta$, 
where $k$ and $l$ are given by
\begin{equation}    \label{eq3.5}
k=\frac{n-m-2}{2},\quad l=\frac{m}{2}.
\end{equation}
\end{theorem}

\begin{proof}
Let $\nu<k+1$ be a nonnegative integer and 
set for $j=0,\ldots,\nu$
\begin{displaymath}
a_j=\frac{\Gamma(k+1)\Gamma(l)}{2\,\Gamma(k-j+1)\Gamma(l+j+1)}.
\end{displaymath}
Because of $2la_0=1$ and $(l+j)a_j-(k-j+1)a_{j-1}=\,0$ 
for $j=1,\ldots,\nu$, then
\begin{displaymath}
\frac{\diff{}}{\diff t}\,\bigg\{
\sum_{j=0}^\nu a_j\,(1-t^2)^{k-j}\,t^{2(l+j)}\bigg\} =\,
(1-t^2)^k\,t^{2l-1}+R_\nu(t)
\end{displaymath}
holds on the set of all $t$ between $-1$ and $1$, 
where the remainder is given by
\begin{displaymath}
R_\nu(t)=2\,(\nu-k)\,a_\nu\,(1-t^2)^{k-\nu-1}\,t^{2(l+\nu)+1}.
\end{displaymath}
If $k$ is, as in the present case, itself an 
integer and $\nu=k$ is chosen, this remainder 
vanishes. As the function possesses, in terms 
of the given $k$ and $l$, the representation
\begin{displaymath}
\psi\bigg(\frac{\delta}{\sqrt{1-\delta^2}}\bigg) =\;
\frac{2\,\Gamma(k+l+1)}{\Gamma(k+1)\Gamma(l)}\,
\int_0^{\delta}(1-t^2)^k\,t^{2l-1}\dt,
\end{displaymath} 
its derivative and that of the right-hand side 
of (\ref{eq3.4}) thus coincide. As both sides 
of this equation take at $\delta=0$ the value 
zero, this proves the proposition.
\end{proof}

For \mbox{$m=128$} and \mbox{$n=256$}, for example, the 
function (\ref{eq3.4}) takes for \mbox{$\delta\leq 1/4$} 
values less than \mbox{$1.90\cdot 10^{-42}$}, and even 
for \mbox{$\delta\leq 1/2$} still values less than 
\mbox{$6.95\cdot 10^{-10}$}. For \mbox{$m=1024$} and 
\mbox{$n=2048$}, these values fall to 
\mbox{$2.68\cdot 10^{-66}$} and 
\mbox{$3.54\cdot 10^{-325}$}, that is, de facto to 
zero. This clearly demonstrates the announced effect.

The coefficients in (\ref{eq3.4}) are rational 
numbers. They can be calculated recursively 
starting from the last one, which takes 
independent of $m$ and $n$ or $k$ and $l$ 
the value one.
Things become particularly simple when $m$ 
and $n$ are both even and $k$ and $l$ are 
then both integers. The representation 
(\ref{eq3.4}) then turns into the sum 
\begin{equation}    \label{eq3.6}
\psi\bigg(\frac{\delta}{\sqrt{1-\delta^2}}\bigg) 
=\;\sum_{j=l}^{k+l}\,\binom{k+l\,}{j}
(1-\delta^2)^{k+l-j}\,\delta^{2j}
\end{equation}
of Bernstein polynomials of order $k+l=n/2-1$ 
in the variable $\delta^2$. One can even go 
a step further. Let $\chi$ be a step function 
with values $\chi(t)=0$ for $t<m/n$ and 
$\chi(t)=1$ for $t>m/n$. The representation 
can then be considered as the approximation
\begin{equation}    \label{eq3.7}
\psi\bigg(\frac{\delta}{\sqrt{1-\delta^2}}\bigg) 
=\;\sum_{j=0}^{k+l}\,
\chi\bigg(\frac{j}{k+l}\bigg)\binom{k+l\,}{j}
(1-\delta^2)^{k+l-j}\,\delta^{2j}
\end{equation}
of $\chi$ by the Bernstein polynomial of order 
$n/2-1$ in the variable $\delta^2$. If the 
ratio of $m$ and $n$ is kept fixed, these 
polynomials tend at all points $\delta$ 
less than
\begin{equation}    \label{eq3.8}
\delta_0=\sqrt{\frac{m}{n}}
\end{equation}
to zero and at all points $\delta>\delta_0$ to one. 
The convergence is even uniform outside every open 
neighborhood of jump position $\delta_0$. This 
follows from the theory of Bernstein polynomials
\cite{Lorentz}, but also from the considerations in 
the next section and is from the random projection 
theorem a known fact. Figure~\ref{fig1} reflects 
this behavior.

\begin{figure}[t]   \label{fig1}
\includegraphics[width=0.93\textwidth]{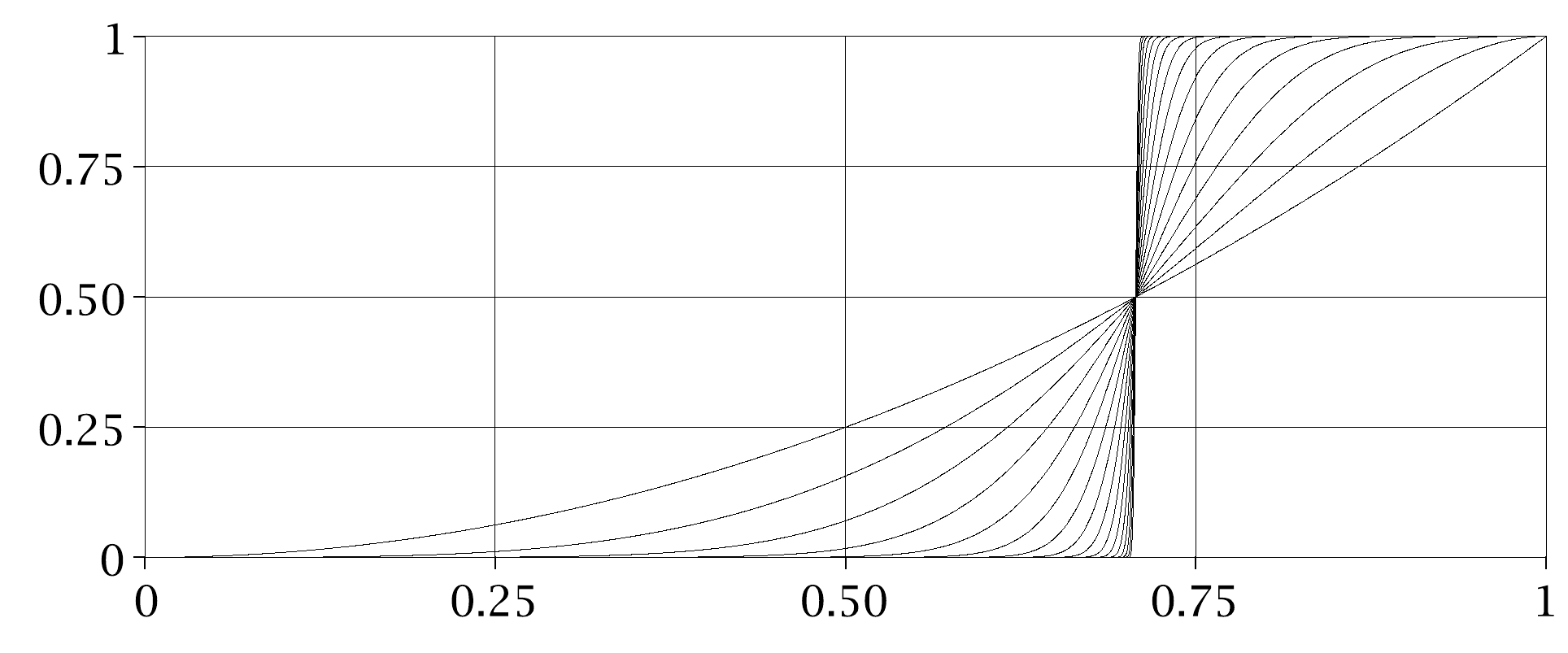}
\caption{The volume ratio \rmref{eq2.8} as function of
 $0\leq\delta<1$ for $m=2^k$, $k=1,\ldots,16$, and $n=2m$}
\end{figure}

If the difference $n-m$ of the dimensions is odd, 
the arguments from the proof of Theorem~\ref{thm3.2}
lead to a representation with an integral remainder
that can, however, in the given context almost
always be neglected.

\begin{theorem}     \label{thm3.3}
If the difference $n-m\geq 3$ of the two dimensions
is odd, if the quantities $k$ and~$l$ are 
defined as in the previous theorem, and 
if $\nu=k-1/2$ is set, 
\begin{equation}    \label{eq3.9}
\psi\bigg(\frac{\delta}{\sqrt{1-\delta^2}}\bigg) 
=\;\sum_{j=0}^{\nu}\,
\frac{\Gamma(k+l+1)}{\Gamma(k-j+1)\Gamma(l+j+1)}\,
(1-\delta^2)^{k-j}\,\delta^{2(l+j)}
+\Delta(\delta)
\end{equation}
holds, where the remainder possesses the
integral representation
\begin{equation}    \label{eq3.10}
\Delta(\delta)=
\frac{2}{\sqrt{\pi}}\frac{\Gamma(n/2)}{\Gamma((n-1)/2)}
\int_0^\delta\frac{t^{n-2}}{\sqrt{1-t^2}}\dt
\end{equation}
and satisfies the estimate 
$0\leq\Delta(\delta)\leq\delta^{n-1}$ on its 
interval of definition.
\end{theorem}

\begin{proof}
By a simple substitution one obtains
\begin{displaymath}
\frac{\Delta(\delta)}{\delta^{n-1}}=
\frac{2}{\sqrt{\pi}}\frac{\Gamma(n/2)}{\Gamma((n-1)/2)}
\int_0^1\frac{t^{n-2}}{\sqrt{1-\delta^2t^2}}\dt
\leq \Delta(1).
\end{displaymath}
Because of $\Delta(1)=1$, the estimate
$\Delta(\delta)\leq\delta^{n-1}$ follows.
\end{proof}


\section{On the limit behavior for high space dimensions}
\label{sec4}

\setcounter{equation}{0}
\setcounter{theorem}{0}

\setcounter{figure}{1}

In this section, we derive bounds for the 
area ratios (\ref{eq1.2}) and the volume ratios 
(\ref{eq2.6}), respectively, with the aim to 
understand their behavior when the dimensions 
tend to infinity. The starting point is a result 
on general matrices. Its proof is based on the 
Markov inequality and once again on the 
separability of Gauss functions. 

\begin{theorem}     \label{thm4.1}
Let $0<\sigma_1\leq\sigma_2\leq\cdots\leq\sigma_m$ 
be the singular values of the matrix~$A$ under 
consideration. The volume ratio \rmref{eq2.6} can  
then be estimated as
\begin{equation}    \label{eq4.1}
\frac
{\lambda\big(\big\{x\,\big|\,\|Ax\|<\delta\,\|A\|\|x\|,\,\|x\|\leq R\big\}\big)}
{\lambda\big(\big\{x\,\big|\,\|x\|\leq R\big\}\big)} 
\leq \min_t X(t)
\end{equation}
by the minimum of the strictly convex function
\begin{equation}    \label{eq4.2}
X(t)=\bigg(\,\prod_{k=1}^m
\frac{1}{1-\delta^2\sigma_m^2\,t+\sigma_k^2\,t}\bigg)^{1/2}
\bigg(\frac{1}{1-\delta^2\sigma_m^2\,t}\bigg)^{(n-m)/2}
\end{equation}
over its interval $0\leq t<1/(\delta^2\sigma_m^2)$ 
of definition.
\end{theorem}

\begin{proof}
We can restrict ourselves by Lemma~\ref{thm2.3} as 
before again to the diagonal matrix $\Sigma$ with the 
entries $\sigma_1,\ldots,\sigma_m$. The characteristic 
function $\chi$ of the set of all $x$ for which 
\mbox{$\|\Sigma x\|<\delta\,\|\Sigma\|\|x\|$} holds 
satisfies, for any $t>0$, the crucial estimate
\begin{displaymath}      
\chi(x)<
\exp\big(\,t\,\big(\delta^2\|\Sigma\|^2\|x\|^2-\|\Sigma x\|^2\big)\big)
\end{displaymath}  
by a product of univariate functions.
By Lemma~\ref{thm2.1}, the subsequent remark, and 
Lemma~\ref{thm2.2}, the volume ratio (\ref{eq2.7}) 
can therefore be estimated by the integral
\begin{displaymath}      
\Big(\frac{1}{\sqrt{\pi}}\Big)^n\!\int
\exp\big(\,t\,\big(\delta^2\|\Sigma\|^2\|x\|^2-\|\Sigma x\|^2\big)\big)
\exp\big(-\|x\|^2\big)\dx
\end{displaymath}    
that remains finite for all $t$ in the given interval.
It splits into a product of one-dimensional integrals 
and takes, for given $t$, the value $X(t)$.  All 
even-order order derivatives of the function $X(t)$ 
are greater than zero, as follows by differentiation 
under the integral sign. The function is therefore,
in particular, strictly convex.
\end{proof}

Because the second-order derivative of $X(t)$ is 
greater than zero and $X'(t)$ tends to infinity as 
$t$ approaches the right endpoint of the interval, 
the first-order derivative of the function $X(t)$ 
possesses a then also unique zero $t^*>0$ if and 
only if
\begin{equation}    \label{eq4.3}
X'(0)=\frac{n}{2}\,\delta^2\sigma_m^2-\frac{1}{2}\sum_{k=1}^m\sigma_k^2
\end{equation}
is less than zero, or equivalently if $\delta$ 
satisfies the condition
\begin{equation}    \label{eq4.4}    
\bar{\kappa}\delta<\sqrt{\frac{m}{n}}, \quad
\frac{1}{\bar{\kappa}^2}=
\frac{1}{m}\sum_{k=1}^m\Big(\frac{\sigma_k}{\sigma_m}\Big)^2.
\end{equation}   
The strictly convex function $X(t)$ attains then and 
only then its minimum at a point in the interior of 
the interval and there then takes a value less than 
$X(0)=1$. Otherwise, the estimate (\ref{eq4.1}) is 
worthless and $X(t)\geq 1$ for all $t$ in the interval. 

\pagebreak  

The minimum of the function (\ref{eq4.2}) can in 
general only be calculated numerically, say by
some variant of Newton's method, and cannot be 
given in closed form. It is, however, 
comparatively simple to estimate this minimum 
from above and below.

\begin{lemma}       \label{thm4.2}
Let $\kappa=\sigma_m/\sigma_1$ again 
be the condition number of the matrix and let~$\xi$ 
be the square root of the dimension ratio $m/n$. 
If $\kappa\delta<\xi$, then the estimate
\begin{equation}    \label{eq4.5}
\min_t X(t)\leq\bigg(\,\frac{\kappa\delta}{\xi}\,
\bigg(\frac{1-\kappa^2\delta^2}{1-\xi^2}\bigg)^\gamma\;
\bigg)^m,
\quad \gamma=\frac{1-\xi^2}{2\xi^2},
\end{equation}
holds for the minimum of the function \rmref{eq4.2}.
Under the for condition numbers $\kappa>1$ weaker 
condition \rmref{eq4.4}, conversely the lower estimate
\begin{equation}    \label{eq4.6}
\min_t X(t)\geq\bigg(\,\frac{\bar{\kappa}\delta}{\xi}\,
\bigg(\frac{1-\bar{\kappa}^2\delta^2}{1-\xi^2}\bigg)^\gamma\;
\bigg)^m
\end{equation}
holds. For orthogonal projections, in both cases 
equality holds and it is
\begin{equation}    \label{eq4.7}
\min_t X(t)=\bigg(\,\frac{\delta}{\xi}\,
\bigg(\frac{1-\delta^2}{1-\xi^2}\bigg)^\gamma\;
\bigg)^m.
\end{equation}
\end{lemma}

\begin{proof}
The function (\ref{eq4.2}) reads in the case
of orthogonal projections, that is, if all
singular values of the matrix take the value 
$\sigma_k=1$, as
\begin{displaymath}
X(t)=
\bigg(\frac{1}{1-\delta^2 t+t}\bigg)^{m/2}
\bigg(\frac{1}{1-\delta^2 t}\bigg)^{(n-m)/2}.
\end{displaymath}
It attains its minimum at the point
\begin{displaymath}
t=\frac{\xi^2-\delta^2}{(1-\delta^2)\delta^2}
\end{displaymath}
in its interval $0<t<1/\delta^2$ of definition 
and takes there the value (\ref{eq4.7}). In 
the general case, the function (\ref{eq4.2}) 
satisfies, because of $\sigma_1\leq\sigma_k$,
the estimate
\begin{displaymath}
X(t)\,\leq\,\bigg(
\frac{1}{1-\delta^2\sigma_m^2\,t+\sigma_1^2\,t}\bigg)^{m/2}
\bigg(\frac{1}{1-\delta^2\sigma_m^2\,t}\bigg)^{(n-m)/2}.
\end{displaymath}
The upper estimate (\ref{eq4.5}) thus follows 
minimizing the right-hand side  as above as a
function of $t'=\sigma_1^2\,t$. The proof of 
the lower estimate is a little bit more involved. 
Since the geometric mean can be estimated by 
the arithmetic mean, 
\begin{displaymath}
\bigg(\,\prod_{k=1}^m
\big(1-\delta^2\sigma_m^2\,t+\sigma_k^2\,t\big)\bigg)^{1/m}
\!\leq\;
\frac{1}{m}\sum_{k=1}^m
\big(1-\delta^2\sigma_m^2\,t+\sigma_k^2\,t\big)
\end{displaymath}
holds. 
By the definition (\ref{eq4.4}) of $\bar{\kappa}$, 
this leads to the estimate
\begin{displaymath}
X(t)\,\geq\, 
\bigg(
\frac{1}{1-\delta^2\sigma_m^2\,t+\bar{\kappa}^{-2}\sigma_m^2\,t}\bigg)^{m/2}
\bigg(\frac{1}{1-\delta^2\sigma_m^2\,t}\bigg)^{(n-m)/2}.
\end{displaymath}
Minimizing the right-hand side as a function 
of $t'=\bar{\kappa}^{-2}\sigma_m^2\,t$,
one gets (\ref{eq4.6}).
\end{proof}

The question is how tight the derived inclusion
for the minimum of the function (\ref{eq4.2})
is in nontrivial cases, for condition numbers 
$\kappa>1$. The answer is that there is practically 
no room for improvement without additional 
conditions to the singular values. Consider a 
sequence of matrices with fixed dimension ratios 
$\xi^2=m/n$ and fixed condition number $\kappa$ 
and let $\kappa\delta<\xi$. Assume that 
$\bar{\kappa}$ tends to $\kappa$ as $m$ goes 
to infinity. This is, for example, the case 
if $\sigma_k=\sigma_1$ for $k=1,\ldots,m-1$. 
The rates
\begin{equation}    \label{eq4.8}
\frac{\bar{\kappa}\delta}{\xi}\,
\bigg(\frac{1-\bar{\kappa}^2\delta^2}{1-\xi^2}\bigg)^\gamma,
\quad \frac{\kappa\delta}{\xi}\,
\bigg(\frac{1-\kappa^2\delta^2}{1-\xi^2}\bigg)^\gamma
\end{equation}
then approach each other arbitrarily as 
$m$ goes to infinity. If additionally
\begin{equation}    \label{eq4.9}
\bar{\kappa}=\kappa-\frac{\kappa_1}{m}+o\Big(\frac{1}{m}\Big)
\end{equation}
holds with some positive constant $\kappa_1$,  
the ratio of the two bounds enclosing the 
minimum of the function (\ref{eq4.2}) tends 
to a limit value greater than zero.

The bound (\ref{eq4.5}) can be simplified, and 
the minimum of the function (\ref{eq4.2}) be 
further estimated in terms of the function
\begin{equation}    \label{eq4.10}
\phi(\vartheta) =
\vartheta\,\exp\bigg(\frac{1-\vartheta^2}{2}\bigg),
\end{equation}
which increases on the interval $0\leq\vartheta\leq 1$ 
strictly, attains at the point $\vartheta=1$ its 
maximum value one, and decreases from there again 
strictly.

\begin{lemma}       \label{thm4.3}
As long as $\kappa\delta$ is less than
the square root $\xi$ of $m/n$, one has
\begin{equation}    \label{eq4.11}
\min_t X(t)\leq\phi\bigg(\frac{\kappa\delta}{\xi}\bigg)^m.
\end{equation}
\end{lemma}

\begin{proof}
Set $\kappa\delta/\xi=\vartheta$ for abbreviation.
The logarithm
\begin{displaymath}
\ln\bigg(\bigg(\frac{1-\kappa^2\delta^2}{1-\xi^2}\bigg)^\gamma\;\bigg)
=\,\frac{1-\xi^2}{2\xi^2}\,\ln\bigg(\frac{1-\vartheta^2\xi^2}{1-\xi^2}\bigg)
\end{displaymath}
then possesses, because of $\vartheta^2\xi^2<1$ and 
$\xi^2<1$, the power series expansion
\begin{displaymath}
\frac {1-\vartheta^2}{2}-\frac12\,\sum_{k=1}^\infty
\bigg(\frac{1-\vartheta^{2k}}{k}-\,\frac{1-\vartheta^{2k+2}}{k+1}\bigg)\xi^{2k}.
\end{displaymath}
Because the series coefficients are for all 
$\vartheta\geq 0$ greater than or equal to 
zero and, by the way, polynomial multiples of 
$(1-\vartheta^2)^2$, the proposition follows 
from (\ref{eq4.5}).
\end{proof}

In the following, we will use the estimate 
(\ref{eq4.11}) for the minimum of the 
function given by (\ref{eq4.2}). The next 
theorem is then a trivial conclusion from 
Theorem~\ref{thm4.1}.

\begin{theorem}     \label{thm4.4}              
Let $n>m$, let $A$ be an $(m\times n)$-matrix of full
rank $m$, let $\kappa$ be the condition number of 
$A$, and let $\xi$ be the square root of the
dimension ratio $m/n$. If $\kappa\delta$ is less 
than $\xi$, then for all radii $R>0$ one has
\begin{equation}    \label{eq4.12}
\frac
{\lambda\big(\big\{x\,\big|\,\|Ax\|<\delta\,\|A\|\|x\|,\,\|x\|\leq R\big\}\big)}
{\lambda\big(\big\{x\,\big|\,\|x\|\leq R\big\}\big)} 
\leq \phi\bigg(\frac{\kappa\delta}{\xi}\bigg)^m.
\end{equation}
\end{theorem}

Consider a sequence of matrices with dimension 
ratios $m/n\geq\delta_0^2$ and condition 
numbers $\kappa\leq\kappa_0$. The volume 
ratios (\ref{eq2.6}) tend then for 
$\kappa_0\delta<\delta_0$ not slower than
\begin{equation}    \label{eq4.13}
\sim\,\phi\bigg(\frac{\kappa_0\delta}{\delta_0}\bigg)^m
\end{equation}
to zero as the dimensions go to infinity.
Under suitable conditions to the singular 
values, considerable improvements are possible.
In extreme cases, such as in Lemma~\ref{thm2.7}, 
the volume ratios (\ref{eq2.6}) can essentially 
be estimated as those for orthogonal projections
and the potentially devastating influence of 
the condition number vanishes.

\begin{theorem}     \label{thm4.5}
Let $n>m$ and let $A$ be a nonvanishing $(m\times n)$-matrix  
with singular values $\sigma_k=\sigma_m$ for $k>m_0$. If 
one sets $m'=m-m_0$ and $\xi'$ is the square root of $m'/n$, 
the volume ratio \rmref{eq2.6} satisfies then for 
$0\leq\delta<\xi'$  the estimate
\begin{equation}    \label{eq4.14}
\frac
{\lambda\big(\big\{x\,\big|\,\|Ax\|<\delta\,\|A\|\|x\|,\,\|x\|\leq R\big\}\big)}
{\lambda\big(\big\{x\,\big|\,\|x\|\leq R\big\}\big)} 
\leq \phi\bigg(\frac{\delta}{\xi'}\bigg)^{m'}.
\end{equation}
\end{theorem}

The proof results from Lemma~\ref{thm2.7},
Theorem~\ref{thm4.1}, and Lemma~\ref{thm4.3}.

Theorem~\ref{thm4.4} possesses a counterpart that deals
with values $\delta$ greater than the square root of
the ratio of the dimensions $m$ and $n$.

\begin{theorem}     \label{thm4.6}
Let $A$ be a nonvanishing $(m\times n)$-matrix and 
let $\xi$ be the square root of the dimension
ratio $m/n$. For $\xi<\delta\leq 1$ then one has 
\begin{equation}    \label{eq4.15}
\frac
{\lambda\big(\big\{x\,\big|\,\|Ax\|\geq\delta\,\|A\|\|x\|,\,\|x\|\leq R\big\}\big)}
{\lambda\big(\big\{x\,\big|\,\|x\|\leq R\big\}\big)} 
\leq \phi\bigg(\frac{\delta}{\xi}\bigg)^m.
\end{equation}
\end{theorem}

\begin{proof}
We can restrict ourselves again to diagonal matrices 
$A=\Sigma$. Let $P$ be the matrix that extracts from 
a vector in $\mathbb{R}^n$ its first $m$ components.
As $\|\Sigma x\|\leq\|\Sigma\|\|Px\|$ and 
$\|\Sigma\|>0$, the given volume ratio can 
then be estimated by the volume ratio
\begin{displaymath}  
\frac
{\lambda\big(\big\{x\,\big|\,\|Px\|\geq\delta\,\|x\|,\,\|x\|\leq R\big\}\big)}
{\lambda\big(\big\{x\,\big|\,\|x\|\leq R\big\}\big)}.
\end{displaymath}
As in the proof of Theorem~\ref{thm4.1}, we can 
estimate this volume ratio for sufficiently small 
positive values $t$ by the integral
\begin{displaymath}
\Big(\frac{1}{\sqrt{\pi}}\Big)^n\!\int
\exp\big(\,t\,\big(\|Px\|^2-\delta^2\|x\|^2\big)\big)
\exp\big(-\|x\|^2\big)\dx.
\end{displaymath}
This integral splits into a product of 
one-dimensional integrals and takes 
the value
\begin{displaymath}
\bigg(\frac{1}{1+\delta^2 t-t}\bigg)^{m/2}
\bigg(\frac{1}{1+\delta^2 t}\bigg)^{(n-m)/2},
\end{displaymath}
which attains, for $\delta<1$, on the interval 
$0<t<1/(1-\delta^2)$ its minimum at
\begin{displaymath}
t=\frac{\delta^2-\xi^2}{(1-\delta^2)\delta^2}. 
\end{displaymath}
It takes at this point $t$ again the value
\begin{displaymath}
\bigg(\,\frac{\delta}{\xi}\,
\bigg(\frac{1-\delta^2}{1-\xi^2}\bigg)^\gamma\;
\bigg)^m,  \quad \gamma=\frac{1-\xi^2}{2\xi^2}.
\end{displaymath}
This leads as in the proof of Lemma~\ref{thm4.3} 
to the estimate (\ref{eq4.15}). As the set of 
all $x$ for which $\|Px\|=\|x\|$ holds has 
measure zero, (\ref{eq4.15}) remains true for 
$\delta=1$.
\end{proof}

For $(m\times n)$-matrices $A$ with dimension 
ratio $m/n\leq\delta_0^2$, the volume ratios 
(\ref{eq2.6}) tend therefore on the interval 
$\delta_0<\delta\leq 1$ pointwise and on its 
closed subintervals uniformly and exponentially 
to one as $m$ goes to infinity. For sequences 
of matrices for which the ratio $m/n$ of their 
dimensions tends to zero, the volume ratios 
(\ref{eq2.6}) hence tend, for all $\delta>0$, 
pointwise to one. This has, however, often 
less severe implications than it might first 
appear. This is demonstrated by the example 
of the matrices $A=T^t$ from section~\ref{sec2}, 
whose dimensions (\ref{eq2.13}) were
\begin{equation}    \label{eq4.16}
m=3\times N, \quad n=3\times\frac{N(N+1)}{2}.
\end{equation}
Figure~\ref{fig2} shows the bounds for the 
volume ratios (\ref{eq2.6}) resulting from 
the application of Lemma~\ref{thm2.7} to 
these matrices as functions of $\delta<1$ 
for $N$ ranging from $4$ to $32$, or, in 
the framework of quantum mechanics, for 
systems with up to $32$ electrons.

\begin{figure}[t]   \label{fig2}
\includegraphics[width=0.93\textwidth]{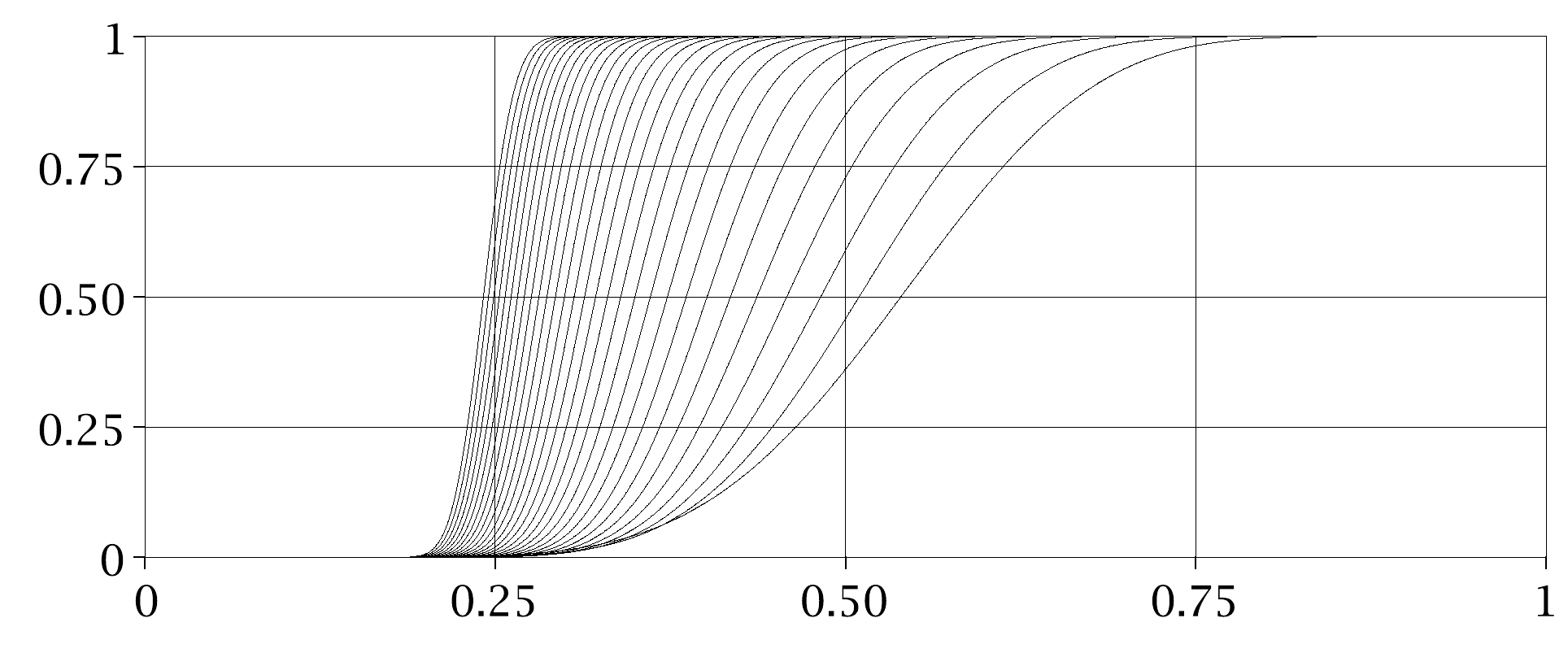}
\caption{The bounds for the volume ratios \rmref{eq2.6}  for the 
 example from section~{\rm \ref{sec2}} for $N=4,\ldots,32$.}
\end{figure}

We finally consider the case of orthogonal projections
from the $\mathbb{R}^n$ to the $\mathbb{R}^m$, that 
is, of $(m\times n)$-matrices $P$ with one as the only
singular value. The condition number of such matrices
is $\kappa=1$, and their norm is $\|P\|=1$. If the 
dimension ratios $m/n$ tend to $\delta_0^2$, or even 
remain as in Figure~\ref{fig1} fixed, the corresponding 
volume ratios (\ref{eq2.8}) tend therefore to a step 
function with jump discontinuity at~$\delta_0$. This 
observation is widely equivalent to the random projection 
theorem. Let $\xi$ be again the square root of $m/n$. 
For a randomly chosen vector $x$, the probability that
\begin{equation}    \label{eq4.17}
(1-\varepsilon)\xi\,\|x\|\leq\|Px\|<(1+\varepsilon)\xi\,\|x\|
\end{equation}
holds is then $F((1+\varepsilon)\xi)-F((1-\varepsilon)\xi)$, 
with the at least for even-numbered differences of the 
dimensions explicitly known distribution function
\begin{equation}    \label{eq4.18}
F(\delta)=\psi\bigg(\frac{\delta}{\sqrt{1-\delta^2}}\bigg),
\end{equation}
and tends exponentially to one as $m$ goes to 
infinity. This means that the orthogonal projection 
of a randomly chosen unit vector $x\in\mathbb{R}^n$ 
onto a given subspace of high dimension $m$ 
possesses with high probability a norm 
\begin{equation}    \label{eq4.19}
\approx\sqrt{\frac{m}{n}}.
\end{equation}
A lower bound for this probability depending only 
on the dimension $m$ but not on the dimension $n$ 
can be derived from the estimates (\ref{eq4.12}) 
and (\ref{eq4.15}). Because of
\begin{equation}    \label{eq4.20}
\phi(1\pm\varepsilon)<\exp(-c\,\varepsilon^2),
\quad c=-\ln(\phi(2)),
\end{equation}
for values $0<\varepsilon<1$, the probability that 
(\ref{eq4.17}) holds is in any case greater than 
\begin{equation}    \label{eq4.21}
1-2\exp(-c\,\varepsilon^2m)
\end{equation}
and the random projection theorem recovered.


\bibliographystyle{siamplain}
\bibliography{references}


\end{document}